\documentclass[12pt]{amsart}
\usepackage{amssymb}
\usepackage[all]{xy}
\usepackage{color}

\input xy
\xyoption{all}

\usepackage{amscd}
\newtheorem{lemma}{Lemma}[section]

\newtheorem{theorem}[lemma]{Theorem}

\theoremstyle{definition}
\newtheorem{remark}[lemma]{Remark}
\newtheorem{definition}[lemma]{Definition}

\newtheorem{example}[lemma]{Example}

\DeclareMathOperator{\Hom}{Hom}
\DeclareMathOperator{\Ext}{Ext}
\DeclareMathOperator{\End}{End}
\DeclareMathOperator{\abs}{abs}

\begin{document}

\title{Strongness of companion bases for cluster-tilted algebras of finite type}

\author{Karin Baur}
\address{Institut f\"{u}r Mathematik und Wissenschaftliches Rechnen, Universit\"{a}t Graz, Heinrichstrasse
36, A-8010 Graz, Austria} \email{baurk@uni-graz.at}

\author{Alireza Nasr-Isfahani}
\address{Department of Mathematics,
University of Isfahan,
P.O. Box: 81746-73441, Isfahan, Iran and School of Mathematics, Institute for Research in Fundamental Sciences (IPM), P.O. Box: 19395-5746, Tehran,
Iran}
\email{nasr$_{-}$a@sci.ui.ac.ir / nasr@ipm.ir}



\begin{abstract} For every cluster-tilted algebra of simply-laced Dynkin type we provide
a companion basis which is strong,
i.e.\ gives the set of dimension vectors of the finitely generated indecomposable
modules for the cluster-tilted algebra.
This shows in particular that every companion basis of a cluster-tilted algebra of simply-laced Dynkin type is strong.
Thus we give a proof of Parsons's conjecture.
\end{abstract}

\maketitle


\section{Introduction}
Cluster algebras were introduced and first investigated by Fomin and
Zelevinsky \cite{FZ1} in order to better understand the dual
canonical basis of the quantised enveloping algebra of a finite
dimensional semisimple Lie algebra. Today cluster algebras are
connected to various subjects including representation theory of
finite dimensional algebras, Poisson geometry, algebraic geometry,
knot theory, Teichm\"{u}ller theory, dynamical systems, mathematical
physics, combinatorics, $\dots$. Cluster categories were introduced
in \cite{BMRRT} (also for type $A_n$ in \cite{CCS}) as a categorical
model for better understanding of the cluster algebras. The
cluster-tilted algebras were introduced by Buan, Marsh and Reiten
\cite{BMR1} have a key role in  the study of cluster categories.
Also an important connection between cluster algebras and
cluster-tilted algebras was established in \cite{BMR2} and
\cite{CCS2}. More precisely it was proved in \cite{BMR2} and
\cite{CCS2} that the quivers of the cluster-tilted algebras of a
given simply-laced Dynkin type are precisely the quivers of the
exchange matrices of the cluster algebras of that type.

Fomin and Zelevinsky \cite{FZ2} classified the cluster algebras of
finite type. Their classification is identical to the Cartan-Killing
classification of semisimple Lie algebras. The conditions of their
classification is hard to check in general. For solving this
problem, Barot, Geiss and Zelevinsky \cite{BGZ} considered the
positive quasi-Cartan matrices which are related to the Cartan
matrices. The main Theorem of \cite{BGZ} shows that the exchange
matrix associated to the quiver of any cluster-tilted algebra of
simply-laced Dynkin type have a positive quasi-Cartan companion.
More precisely they showed that there exists a $\mathbb{Z}$-basis of
roots of the integral root lattice of the corresponding root system
of simply-laced Dynkin type such that the matrix of inner products
associated to this basis creates the positive quasi-Cartan companion.
Parsons in \cite{P} and \cite{P1} called such a $\mathbb{Z}$-basis a
companion basis. Parsons used companion bases to study
indecomposable modules over cluster-tilted algebras. Note that the
knowledge of the indecomposable modules over cluster-tilted algebras is not only
important in the representation theory of cluster-tilted algebras
but also has applications in the corresponding cluster algebra.
Nakanishi and Stella \cite{NS} showed that the set of all dimension
vectors of the indecomposable modules over a cluster-tilted algebra
of finite type, the set of the non-initial $d$-vectors of the corresponding
cluster algebra and the set of positive $c$-vectors of the corresponding
cluster algebra coincide.

For any companion basis of a cluster-tilted algebra of simply-laced
Dynkin type, Parsons defined a collection of positive vectors. Any
positive root of the corresponding root system of simply-laced
Dynkin type is a linear combination of elements of the companion basis with
integer coefficients. That collection is the set of the absolute
values of these coefficients. A companion basis of a cluster-tilted
algebra of simply-laced Dynkin is called strong \cite{P}, \cite{P1}
if the associated collection of positive vectors are precisely the
dimension vectors of the finitely generated indecomposable modules
over the given cluster-tilted algebra. Parsons proved that for any
cluster-tilted algebra of Dynkin type $A$, any companion basis is
strong. He also conjectured that any companion basis of any
cluster-tilted algebra of simply-laced Dynkin type is strong
(Conjecture 6.3 of \cite{P}).

On the other hand Ringel \cite{R} studied finitely generated indecomposable modules over cluster-concealed algebras (i.e. cluster-tilted algebras where the associated cluster-tilting object corresponds to a preprojective tilting module). Note that cluster-tilted algebras of simply-laced Dynkin type are cluster-concealed algebras, but there are many cluster-concealed algebras which are tame or wild. Ringel \cite{R} used a theorem due to Assem et al. \cite{ABS} and Zhu \cite{Z}, which gives an equivalent definition of cluster-tilted algebras. According to that theorem, any cluster-tilted algebra is isomorphic to the relation-extension of some tilted algebra. By using the tilting functor
Ringel defined a linear bijection map $g$ between the Grothendieck group of the corresponding hereditary algebra
$K_0(H)$ and the Grothendieck group of the tilted algebras $K_0(B)$. He proved that the dimension vectors of indecomposable modules over a
cluster-concealed algebra are precisely the vectors $\abs(x)$ of absolute values
with $x\in \Phi_B$, where $\Phi_B$ is the image of the root system $\Phi_H$ of $H$ under $g$ and where
for each $x=(x_1, \cdots, x_n)\in \mathbb{Z}^n$, $\abs(x)=(|x_1|, \cdots, |x_n|)$. He also proved that indecomposable modules over cluster-tilted algebras of simply-laced Dynkin type are uniquely determined by their dimension vectors.

Given a cluster-tilted algebra of simply-laced Dynkin type, according to Ringel's Theorem there
exist $\{x_1, x_2, \cdots, x_n\}\subseteq \Phi_H$ such that $\abs(g(x_i))=e_i$  for each $i$, where $\{e_1, \cdots, e_n\}$ is the canonical basis of
$\mathbb{Z}^n$.
The goal of this paper is to show that the set $\{x_1,x_2,\cdots,x_n\}$ is a companion basis for
the cluster-tilted algebra and that it is strong.
This result not only illustrates the connection between Ringel's results \cite{R} and Parsons's
results \cite{P}, \cite{P1}, it also proves Parsons's conjecture (Conjecture 6.3 of \cite{P}).
Before proving our results in Section~\ref{sec:results} we provide the necessary background in the
following section.


\section{Preliminaries}
 Let $\mathbb{F}=\mathbb{Q}(u_1, \cdots, u_n)$
be the field of rational functions in $n$ indeterminates. Let
$\textbf{x} \subseteq \mathbb{F}$ be a transcendence basis over
$\mathbb{Q}$, and let $C = (c_{xy})_{x,y\in \textbf{x}}$ be an
$n\times n$ sign-skew-symmetric integer matrix with rows and columns
indexed by $\textbf{x}$. The pair $(\textbf{x},C)$ is called a seed.
Given such a seed, and an element $z\in \textbf{x}$, define a new
element $z'\in \mathbb{F}$ via the exchange relation:  $$zz' =
\prod_{x\in \textbf{x}, c_{xz}>0} x^{c_{xz}} + \prod_{x\in
\textbf{x}, c_{xz}<0} x^{-c_{xz}} .$$
Let $\textbf{x}' = \textbf{x}
\cup \{z'\} \backslash \{z\}$, it is a new transcendence basis of
$\mathbb{F}$. Let $C'$ be the mutation of the matrix $C$ in direction
$z$:
\begin{align*}
c_{xy}' = \begin{cases}
-c_{xy} & \text{if }\ x=z\ \text{or}\ y=z,\\
c_{xy}+\frac{|c_{xz}|c_{zy}+c_{xz}|c_{zy}|}{2}& \text{otherwise. }\
\end{cases}
\end{align*}
 The row and column
labeled $z$ in $C$ are relabeled $z'$ in $C'$. The pair
$(\textbf{x}', C')$ is called the mutation of the seed $(\textbf{x}, C)$
in direction $z$. Let $\mathcal{S}$ be the set of seeds obtained by
iterated mutation of $(\textbf{x},C)$. Then the set of cluster
variables is, by definition, the union $\chi$ of all the elements of the transcendence
bases appearing in the seeds in $\mathcal{S}$. These bases are known
as clusters, and the cluster algebra $\mathcal{A}(\textbf{x}, C)$ is
the subring of $\mathbb{F}$ generated by $\chi$. Up to isomorphism
of cluster algebras, it does not depend on the initial choice
$\textbf{x}$ of transcendence basis, so it is just denoted by
$\mathcal{A}_C$. If $\chi$ is finite, the cluster algebra
$\mathcal{A}_C$ is said to be of finite type.

Let $H$ be a hereditary finite dimensional $k$-algebra, where $k$ is
an algebraically closed field and let $\mathcal{D} = D^b({\rm mod} H)$ be
the bounded derived category of finitely generated left $H$-modules
with shift functor $[1]$. Also, let $\tau$ be the $AR$-translation
in $\mathcal{D}$. The cluster category is defined as the orbit
category $\mathcal{C}_H = \mathcal{D} /F$, where $F = \tau^{-1}[1]$.
The objects of $\mathcal{C}_H$ are the same as the objects of
$\mathcal{D}$, but maps are given by $\Hom_{\mathcal{C}_H}(X, Y )
=\bigoplus_{i\in \mathbb{Z}} \Hom_\mathcal{D}(X, F^iY)$. An object $\widetilde{T}$ in
$\mathcal{C}_H$ is called cluster-tilting provided for any object
$X$ of $\mathcal{C}_H$, we have $\Ext^1_{\mathcal{C}_H}(\widetilde{T}, X) = 0$ if
and only if $X$ lies in the additive subcategory ${\rm add}(\widetilde{T})$ of
$\mathcal{C}_H$ generated by $\widetilde{T}$. Let $\widetilde{T}$ be a cluster-tilting
object in $\mathcal{C}_H$. The cluster-tilted algebra associated to
$\widetilde{T}$ is the algebra $\End_{\mathcal{C}_H}(\widetilde{T})^{op}$.

Let $Q$ be a quiver of simply-laced Dynkin type with underlying graph $\Delta$ and $n$ vertices. Let $H=kQ$ and $\widetilde{T}$ be a basic cluster-tilting object in the cluster category $\mathcal{C}_H$. Let $\Lambda=\End_{\mathcal{C}_H}(\widetilde{T})^{op}$ be the corresponding cluster-tilted algebra. Let $\mathcal{A}$ be the cluster algebra of type $\Delta$ and $(\textbf{x},C)$ be the seed in $\mathcal{A}$ corresponding to $\widetilde{T}$. According to Theorem 3.1 of \cite{CCS2} (also Section 6 of \cite{BMR2}), $\Gamma=\Gamma(C)$ is the quiver of $\Lambda$, where $\Gamma(C)$ is the quiver with vertices corresponding to the rows and columns of $C$, and $c_{ij}$ arrows from the vertex $i$ to the vertex $j$ whenever $c_{ij}>0$. Let $V$ be the Euclidean space with positive definite symmetric bilinear form $(-, -)$ and $\Phi_H\subseteq V$ be the root system of Dynkin type $\Delta$.

\begin{definition} (Definitions 4.1, 5.1 and 5.2 of \cite{P})
\begin{itemize}
\item[(1)] The subset $\Psi=\{\gamma_i| 1\leq i\leq n\}\subseteq \Phi_H$ is called a companion basis for $\Gamma=\Gamma(C)$ if satisfies the following properties:
\begin{itemize}
\item[(i)] $\{\gamma_i| 1\leq i\leq n\}$ is a $\mathbb{Z}$-basis for $\mathbb{Z}\Phi_H$,
\item[(ii)] $|(\gamma_i, \gamma_j)|=|c_{ij}|$ for all $i\neq j$, $1\leq i, j\leq n$.
\end{itemize}
\item[(2)] Let $\alpha\in \Phi_H$ and suppose that $\alpha=\sum_{i=1}^nc_i\gamma_i$ with $c_i\in \mathbb{Z}$ for each $i$. $d_\alpha^\Psi$ is defined to be the vector $d_\alpha^\Psi=(|c_1|, \cdots, |c_n|)$.
\item[(3)] The companion basis $\Psi=\{\gamma_i| 1\leq i\leq n\}$ of $\Gamma$ is called strong companion basis if the vectors $d_\alpha^\Psi$ for $\alpha\in \Phi_H^+$ are the dimension vectors of the finitely generated indecomposable $\Lambda$-modules.

\end{itemize}
\end{definition}

Let $\Lambda$ be a cluster-tilted algebra of simply-laced Dynkin type and $\Gamma=\Gamma(C)$ be the quiver of $\Lambda$. Parsons conjectured in Conjecture 6.3 of \cite{P}
that all companion bases for $\Gamma$ are strong. He proved this conjecture for cluster-tilted algebras of simply-laced type $A_n$ (Theorem 5.3 of \cite{P}). Let $\Psi=\{\gamma_i| 1\leq i\leq n\}$ and $\Omega=\{\gamma'_i| 1\leq i\leq n\}$ be two arbitrary
companion bases for $\Gamma$. By Proposition 6.2 of \cite{P}, $\{d_\alpha^\Psi|\alpha\in \Phi_H^+\}=\{d_\alpha^\Omega|\alpha\in \Phi_H^+\}$. This tells us that the existence of a strong companion basis for $\Gamma$ shows that all companion bases of $\Gamma$ are strong.

Assem, Br\"{u}stle and Schiffler \cite{ABS} and Zhu \cite{Z} independently provided a
characterization of
cluster-tilted algebras. They proved that an algebra $\Lambda$ is
cluster-tilted if and only if there exists a tilted algebra $B$ such
that $\Lambda\cong B^c$, where $B^c$ is trivial extension algebra
$B^c = B \ltimes \Ext^2_B(DB,B)$, with $D = \Hom(-, k)$ the
$k$-duality. Recall that a $k$-algebra $B$ is said to be tilted
provided $B$ is the endomorphism ring of a tilting $H$-module $T$,
where $H$ is a finite-dimensional hereditary $k$-algebra. A tilted
algebra $B$ is said to be concealed provided $B$ is the endomorphism
ring of a preprojective tilting $H$-module (i.e. a tilting
$H$-module whose summands are all preprojective). If $B$ is a
concealed algebra, then $B^c$ is called a cluster-concealed algebra.

Let $B=\End_H(T)$ be concealed algebra and $\Lambda=B^c$ the corresponding cluster-concealed algebra and $K_0(H)$ be the Grothendieck group of $H$. We identify $K_0(H)$ with $\mathbb{Z}^n$, where $n$ is the number of isomorphism class of simple $H$-modules. Let $M$ be a finitely generated $H$-module. The dimension vector of $M$ is the vector $\mathbf{dim}\, M\in\mathbb{Z}^n$, whose coefficients are the Jordan-H\"{o}lder multiplicities of $M$. We denote by $<-, ->_H$ the bilinear form on $K_0(H)$ given by
\begin{equation*}
<\mathbf{dim}\,M, \mathbf{dim}\,N>_H=\sum_{i\geq0}(-1)^i \dim_k \Ext_H^i(M, N)
\end{equation*}
for all $H$-modules $M, N$. The corresponding symmetric bilinear form $(-, -)_H$ is defined
as
$(\mathbf{dim}\,M, \mathbf{dim}\,N)_H
=<\mathbf{dim}\,M, \mathbf{dim}\,N>_H+<\mathbf{dim}\,N, \mathbf{dim}\,M>_H$.

Let $T_1, \cdots, T_n$ be indecomposable direct summands of $T$, one from each isomorphism class. Let $g:K_0(H)\rightarrow K_0(B)$ given by $g(x)=(<\mathbf{dim}\, T_i, x>_H)_i$. Then for each $i$, we have $g(\mathbf{dim}\, T_i)=\mathbf{dim}\, G(T_i)$, where $G=\Hom_H(T, -): {\rm mod} H\rightarrow {\rm mod} B$ is a tilting functor. Since $\{\mathbf{dim}\, T_1, \cdots \mathbf{dim}\, T_n\}$ is a basis of $K_0(H)$ and
$\{\mathbf{dim}\, G(T_1), \cdots \mathbf{dim}\, G(T_n)\}$ is a basis of $K_0(B)$, $g$ is a linear bijection.

Let $\Phi_H$ be the root system in $K_0(H)$ corresponding to the underlying graph of the quiver of $H$ and $\Phi_B=g(\Phi_H)$. In Theorem 2 of \cite{R}, Ringel proved that the dimension vectors of the indecomposable $\Lambda$-modules are precisely the vectors $\abs(x)$ with $x\in \Phi_B$.

\begin{remark} Let $\Phi_B^+=g(\Phi_H^+)$, where $\Phi_H^+$ is the set of positive roots of $H$.
For every $x\in \Phi_B$ there exists $a\in \Phi_H$ such that $x=g(a)$. The root $a$ is either positive or negative. If $a$ is negative then $-a\in \Phi_H^+$ and $x=-g(-a)$. Thus $-x\in \Phi_B^+$. Since
$\abs(x)=\abs(-x)$, it is enough to consider positive roots. So according to the Ringel's Theorem,
$\{\mathbf{dim}\, M | M\in ind \Lambda\}=\{\abs(x) | x\in \Phi_B^+\}$.
\end{remark}

In the following example we illustrate the connection between Parsons's results \cite{P} and Ringel's results \cite{R}.

\begin{example}\label{ex1} Consider the quiver $\Gamma$ given by
$$\hskip.5cm 
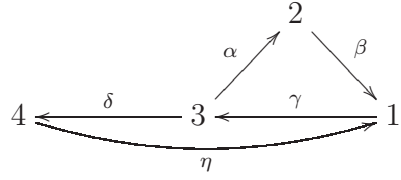
\xymatrix{
&&&&{2} \ar@{<-}[dl]_{\alpha}\ar[dr]^{\beta}\\
&{4}\ar@/_1pc/[rrrr]_{\eta}&&{3}\ar[ll]_{\delta}&&{1}\ar[ll]_{\gamma}
}
\hskip.5cm$$
In fact $\Gamma$ is the mutation of $Q$ at vertex $3$, where $Q$ is the quiver
$$\hskip.5cm 
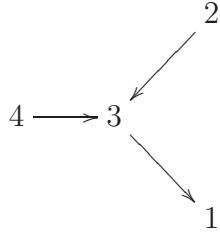
\xymatrix{
&&{2} \ar[dl]\\
{4}\ar[r]&{3}\ar[rd]&\\
&&{1}} \hskip.5cm$$ Thus $\Gamma$ is a quiver of a cluster-tilted
algebra $\Lambda$ of type $D_4$. More precisely $\Lambda\cong
k\Gamma/I$, where $I$ is an admissible ideal of $k\Gamma$ generated
by $\beta\gamma, \gamma\alpha, \eta\gamma, \gamma\delta,
\alpha\beta-\delta\eta$. By Theorem 6.1 of \cite{P},
$\Psi=\{\alpha_1, \alpha_2+\alpha_3, \alpha_3, \alpha_3+\alpha_4\}$
is a companion basis for $\Gamma$, where $\{\alpha_1, \alpha_2,
\alpha_3, \alpha_4\}$ is a simple system of $\Phi_H$ and $H=kQ$.
Recall that the positive roots $\Phi_H^+$ are
$\{\alpha_1, \alpha_2, \alpha_3, \alpha_4,
\alpha_1+\alpha_3, \alpha_2+\alpha_3, \alpha_3+\alpha_4,
\alpha_1+\alpha_2+\alpha_3, \alpha_1+\alpha_3+\alpha_4,
\alpha_2+\alpha_3+\alpha_4, \alpha_1+\alpha_2+\alpha_3+\alpha_4,
\alpha_1+\alpha_2+\alpha_3+\alpha_3+\alpha_4\}$. In terms of $\Psi$, their dimension
vectors are
$\{d_\alpha^\Psi | \alpha\in\Phi_H^+\}=\{(1, 0, 0, 0), (0, 1, 1, 0),
(0, 0, 1, 0), (0, 0, 1, 1), (1, 0, 1, 0), (0, 1, 0, 0), (0, 0, 0,
1), (1, 1, 0, 0),\\ (1, 0, 0, 1), (0, 1, 1, 1), (1, 1, 1, 1), (1, 1,
0, 1)\}$. On the other hand $\Lambda\cong B\ltimes \Ext^2_B(DB, B)$,
where $B=\End_H(T')$, $T'=P_1\oplus P_2\oplus P'_3\oplus P_4$, $P_1,
P_2, P_3, P_4$ are indecomposable projective $H$-modules and $P'_3$
is the indecomposable $H$-module with the dimension vector $(1, 1,
1, 1)$ (in fact $H\cong \End_H(H)$ and $T'$ is a mutation of $H$ at
$3$). An easy calculation shows that $\Phi_B^+=\{(1, 0, 0, 0), (0,
1, 1, 0), (0, 0, -1, 0), (0, 0, 1, 1), (1, 0, -1, 0), (0, 1, 0, 0),
(0, 0, 0, 1),\\ (1, 1, 0, 0), (1, 0, 0, 1), (0, 1, 1, 1), (1, 1, 1,
1), (1, 1, 0, 1)\}$ and so by the Ringel's Theorem the dimension vectors
of the indecomposable $\Lambda$-modules are precisely $\{\abs(x) |
x\in \Phi_B^+\}=\{(1, 0, 0, 0), (0, 1, 1, 0), (0, 0, 1, 0), (0, 0,
1, 1), (1, 0, 1, 0), (0, 1, 0, 0), (0, 0, 0, 1), (1, 1, 0, 0),\\ (1,
0, 0, 1), (0, 1, 1, 1), (1, 1, 1, 1), (1, 1, 0, 1)\}$. This shows
that $\Psi$ is a strong companion basis for $\Gamma$,
the vectors of absolute values of the images of the
four roots of the companion basis are
$\abs(g(\alpha_1))=(1, 0, 0, 0)$, $\abs(g(\alpha_2+\alpha_3))=(0, 1,
0, 0)$, $\abs(g(\alpha_3))=(0, 0, 1, 0)$ and
$\abs(g(\alpha_3+\alpha_4))=(0, 0, 0, 1)$. In fact we have a companion
basis for $\Gamma$ such that its image under $g$ is the canonical
basis of $K_0(B)=\mathbb{Z}^4$. In the next section we prove this
fact for arbitrary cluster-tilted algebras of simply-laced Dynkin type.
\end{example}


\section{Main Result}\label{sec:results}

Let $\Lambda$ be an arbitrary cluster-tilted algebra of finite type. It is known that there exists $H=kQ$,
$Q$ of Dynkin type $\Delta$, and a cluster-tilting object $\widetilde{T}$ of $\mathcal C_H$ such that
$\Lambda\cong \End_{\mathcal C_H}(\widetilde{T})$. Furthermore, it is known that there exists a tilting
$H$-module $T$ such that $\Lambda\cong B^c$ where $B=\End_H(T)$, cf.~\cite{ABS} or \cite{Z}.

Let $\mathcal{A}$ be a cluster algebra of type $\Delta$, and suppose
that $(\mathbf{x}, C)$ is the seed in $\mathcal{A}$ corresponding to
$\widetilde{T}$. Then $\Gamma=\Gamma(C)$ is the quiver of $\Lambda$ (\cite{BMR2}, \cite{CCS2}).

Let $g:K_0(H)\rightarrow K_0(B)$ be the linear bijection given by
$g(x)=(<\mathbf{dim}\, T_i, x>_H)_i$, and $S_1, S_2, \cdots, S_n$ be
the isomorphism class of simple $\Lambda$-modules. By the Ringel's
Theorem there exist $y_1, \cdots, y_n\in \Phi_B^+$ such that
$\abs(y_i)=\mathbf{dim}\, S_i=e_i$, for each $1\leq i\leq n$. Also for
$1\leq i\leq n$ there exists $x_i\in \Phi_H^+$ such that
$g(x_i)=y_i$.

\begin{theorem}\label{Thm} The subset $\{x_1, \cdots, x_n\}\subseteq \Phi_H^+$
is a companion basis for $\Gamma$.
\end{theorem}

\begin{proof} First we show that $\{x_1, \cdots, x_n\}$ is a
$\mathbb{Z}$-basis for $\mathbb{Z}\Phi_H$. Let $x\in
\mathbb{Z}\Phi_H$, then $g(x)\in \mathbb{Z}\Phi_B$, and so
$g(x)=d_1e_1+\cdots+d_ne_n$, where $d_i\in \mathbb{Z}$ for each $i$
and $\{e_1, \cdots, e_n\}$ is the canonical basis of $K_0(B)$. Since
$e_i=\abs(g(x_i))$ for each $i$,
$g(x)=d_1\varepsilon_1g(x_1)+\cdots+d_n\varepsilon_ng(x_n)$ with
$\varepsilon_i\in \{1, -1\}$ for each $i$. Then
$x=d_1\varepsilon_1x_1+\cdots+d_n\varepsilon_nx_n$, since $g$ is a
linear bijection. Now we show that for any $1\leq i, j\leq n$,
$i\neq j$, we have $|(x_i, x_j)_H|=|c_{ij}|$. We have $|(x_i,
x_j)_H|=|(g(x_i), g(x_j))_B|=|(\varepsilon_ie_i,
\varepsilon_je_j)_B|$ with $\varepsilon_i, \varepsilon_j\in \{1,
-1\}$. Then $|(x_i, x_j)_H|=|(e_i, e_j)_B|=|<e_i, e_j>_B+<e_j,
e_i>_B|$.
Since $H$ is hereditary, ${\rm gl.dim}\, B\leq 2$ and so
$<e_i, e_j>_B=\dim_k \Hom_B(S_i, S_j)-\dim_k \Ext^1_B(S_i, S_j)+\dim_k
\Ext^2_B(S_i, S_j)$. Thus $|(x_i, x_j)_H|=|\dim_k \Hom_B(S_i,
S_j)-\dim_k \Ext^1_B(S_i, S_j)+\dim_k \Ext^2_B(S_i, S_j)+\dim_k
\Hom_B(S_j, S_i)-\\ \dim_k \Ext^1_B(S_j, S_i)+\dim_k \Ext^2_B(S_j, S_i)|$.
Since $i\neq j$, $\Hom_B(S_i, S_j)=\Hom_B(S_j, S_i)=0$ and so $|(x_i,
x_j)_H|=|-\dim_k \Ext^1_B(S_i, S_j)+\dim_k \Ext^2_B(S_i, S_j)-\dim_k
\Ext^1_B(S_j, S_i)+\dim_k \Ext^2_B(S_j, S_i)|$. It is known that the
quivers of cluster-tilted algebras contain no $2$-cycles (i.e.
oriented cycles of length two). Therefore, $|c_{ij}|$ is equal to the
number of arrows in $\Gamma=\Gamma(C)$ from $i$ to $j$ plus the
number of arrows in $\Gamma=\Gamma(C)$ from $j$ to $i$. But $\Gamma$
is the quiver of $B^c$ and according to  Theorem 2.6 of
\cite{ABS} (see also the proof of this theorem), the number of
arrows in the quiver of $B^c$, $Q_{B^c}$ from $i$ to $j$ is equal to
the number of arrows in the quiver of $B$, $Q_{B}$ from $i$ to $j$
plus $\dim_k \Ext^2_B(S_j, S_i)$ additional arrows. Also it is known
that the number of arrows in $Q_{B}$ from $i$ to $j$ equals the
$\dim_k \Ext^1_B(S_i, S_j)$. Hence $|c_{ij}|=\dim_k \Ext^1_B(S_i,
S_j)+\dim_k \Ext^2_B(S_j, S_i)+\dim_k \Ext^1_B(S_j, S_i)+\dim_k
\Ext^2_B(S_i, S_j)$. We claim that $|-\dim_k \Ext^1_B(S_i, S_j)+\dim_k
\Ext^2_B(S_i, S_j)-\dim_k \Ext^1_B(S_j, S_i)+\dim_k \Ext^2_B(S_j,
S_i)|=\dim_k \Ext^1_B(S_i, S_j)+\dim_k \Ext^2_B(S_j, S_i)+\dim_k
\Ext^1_B(S_j, S_i)+\dim_k \Ext^2_B(S_i, S_j)$:

It is known that in
$Q_B$ there are no oriented cycles and in $Q_{B^c}$ there are no
multiple arrows. If $\Ext^1_B(S_i, S_j)\neq 0$, then we have an arrow
form $i$ to $j$ in $Q_B$ and so in $Q_{B^c}$. Then in this case
there is no arrow from  $j$ to $i$ and no other arrow from $i$ to
$j$ in $Q_{B^c}$. Thus $\Ext^2_B(S_i, S_j)=\Ext^2_B(S_j,
S_i)=\Ext^1_B(S_j, S_i)=0$ and our claim follows. If
$\Ext^1_B(S_j, S_i)\neq 0$, a similar argument shows that our claim
holds). Finally in case $\Ext^1_B(S_i,
S_j)=\Ext^1_B(S_j, S_i)=0$ our claim is obvious.

Therefore, we have $|(x_i, x_j)_H|=|c_{ij}|$ for any $1\leq i, j\leq n$,
$i\neq j$ and our result follows.
\end{proof}

Now we are ready to prove our main Theorem.

\begin{theorem}\label{mainThm} Let $\Lambda=\End_{\mathcal{C}_H}(\widetilde{T})$ be a cluster-tilted
algebra of simply laced-Dynkin type. Suppose that $(\mathbf{x}, C)$
is the seed corresponding to $\widetilde{T}$. Then all companion
bases for $\Gamma=\Gamma(C)$ are strong.
\end{theorem}

\begin{proof} By Proposition 6.2 of \cite{P}, it is enough to show
that there exits a strong companion basis for $\Gamma$. According to
Theorem \ref{Thm}, $\Psi=\{x_1, \cdots, x_n\}\subseteq \Phi_H^+$
is a companion basis for $\Gamma$. We show that $\Psi$ is a strong
companion basis for $\Gamma$. Let $\alpha\in \Phi_H^+$.
$\alpha=c_1x_1+\cdots+c_nx_n$, with $c_i\in \mathbb{Z}$ for each $i$.
Then
$g(\alpha)=c_1g(x_1)+\cdots+c_ng(x_n)=c_1\varepsilon_1e_1+\cdots+c_n\varepsilon_ne_n$
with $\varepsilon_i\in \{1, -1\}$, for each $i$. Also $g(\alpha)\in
\Phi_B^+ $ and so by Theorem 2 of \cite{R}, $\abs(g(\alpha))=(|c_1|,
\cdots, |c_n|)$ is a dimension vector of some indecomposable
$\Lambda$-module. But $d^\Psi_\alpha=(|c_1|, \cdots, |c_n|)$ by definition and
hence $\Psi$ is a strong companion basis for $\Gamma$.
\end{proof}

It is known that different tilted algebras $B$ and $B'$ may correspond to the same cluster-tilted algebra $B\ltimes \Ext^2_B(DB, B)$. In that case, the linear bijection $g:K_0(H)\rightarrow K_0(B)$
depends on the choice of the tilted algebra $B$. In the following example we show that
different tilted algebras give a different companion basis.

\begin{example} Consider the cluster-tilted algebra $\Lambda=B\ltimes \Ext^2_B(DB, B)$
of Example \ref{ex1}. Let $B'$ be the tilted algebra given by the quiver
$$\hskip.5cm 
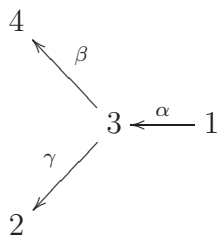
\xymatrix{
{4}&&\\
&{3}\ar[ul]_{\beta}\ar[dl]_{\gamma}&{1}\ar[l]_{\alpha}\\
{2}&&
}
\hskip.5cm$$
bound by $\alpha\beta=0$ and $\alpha\gamma=0$. Then it is easy to see that $B\ltimes \Ext^2_B(DB, B)\cong B'\ltimes \Ext^2_{B'}(DB', B')$ and
$B'=\End_{H'}(T'')$, where
$T''=(0, 0, 0, 1)\oplus (1, 0, 0, 0)\oplus (1, 1, 1, 1)\oplus (0, 1, 0, 0)$
and $H'=kQ'$, where $Q'$ is the quiver
$$\hskip.5cm 
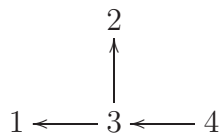
\xymatrix{
&{2}&\\
{1}&{3}\ar[l]\ar[u]&{4}\ar[l]\\
}
\hskip.5cm
$$
Let $g':K_0(H')\rightarrow K_0(B')$ be the linear bijection given by
$g'(x)=(<\mathbf{dim}\, T''_i, x>_{H'})_i$. An easy calculation shows that $\abs(g'(1, 0, 0, 0))=(0, 1, 0, 0)$, $\abs(g'(0, 1, 0, 0))=(0, 0,
0, 1)$, $\abs(g'(0, 0, 1, 0))=(1, 0, 0, 0)$ and
$\abs(g'(0, 0, 1, 1))=(0, 0, 1, 0)$. Then by Theorems \ref{Thm} and \ref{mainThm}, the set $\Psi'=\{\alpha_1, \alpha_2, \alpha_3, \alpha_3+\alpha_4 \}$ is a strong companion basis for $\Lambda$.
\end{example}

\section*{acknowledgements}
Special thanks are due to the referee who read this
paper carefully and made useful comments that improved the presentation of the paper.
Part of this work was carried out during a visit of the second
author to the Institut f\"{u}r Mathematik und Wissenschaftliches
Rechnen, Universit\"{a}t Graz, Austria, with the financial support of the Joint Excellence in Science and Humanities (JESH) program of the Austrian Academy of Sciences. The second author would like to thank the Austrian Academy of Sciences for its support and this host
institution for its warm hospitality. The research of the second
author was in part supported by a grant from IPM (No. 95170419).
The first author acknowledges the support of the Austrian Science Fund (FWF): W1230.

\end{document}